\documentclass[12pt,reqno]{amsart}
\usepackage{amsmath, amsfonts, amssymb, amsthm, hyperref,cite}
\usepackage{bm}
\usepackage{mathrsfs}
\allowdisplaybreaks[4]
\def\bg{\bigg}
\def\({\bg(}
\def\){\bg)}

\def\<{\langle}
\def\>{\rangle}
\def\1{{\bf 1}}

\theoremstyle{plain}
\newtheorem{theorem}{Theorem}

\newtheorem{lemma}{Lemma}
\newtheorem{corollary}{Corollary}
\theoremstyle{definition}

\newtheorem*{Acks}{Acknowledgments}
\theoremstyle{remark}
\newtheorem{remark}{Remark}

\numberwithin{equation}{section}

\begin{document}
\title[Some Congruences Involving Fourth Powers of Generalized Central Trinomial...]{Some Congruences Involving Fourth Powers of Generalized Central Trinomial Coefficients}

\author[Yassine Otmani]{Yassine Otmani}
\address[Yassine Otmani]{Department of Mathematics, USTHB, RECITS Laboratory, El Alia 16111, Po. Box 32 Bab Ezzouar,Algiers, Algeria}
\email{yotmani@usthb.dz \and otmani.yassine@yahoo.fr}

\author[Hac{\`e}ne Belbachir]{Hac{\`e}ne Belbachir}
\address[Hac{\`e}ne Belbachir]{Department of Mathematics, USTHB, RECITS Laboratory, El Alia 16111, Po. Box 32 Bab Ezzouar,Algiers, Algeria}
\email{hbelbachir@usthb.dz \and hacenebelbachir@gmail.com}

\begin{abstract}
Let $ p \ge 5 $ be a prime and let $ b, c \in \mathbb{Z} $. Denote by $ T_k(b,c) $ the generalized central trinomial coefficient, i.e., the coefficient of $ x^k $ in $ (x^2 + bx + c)^k $. In this paper, we establish congruences modulo $ p^3 $ and $ p^4 $ for sums of the form
$$
\sum_{k=0}^{p-1} (2k+1)^{2a+1}\,\varepsilon^{k}\,\frac{T_k(b,c)^4}{d^{2k}},
$$
where $ a \in  \left\lbrace 0,1\right\rbrace  $, $ \varepsilon \in \{1,-1\} $, and $ d = b^2 - 4c $ satisfies $ p \nmid d $. In particular, for the special case $ b = c = 1 $, we show that
\begin{align*}
\sum_{k=0}^{p-1}\left( 2k+1\right) ^{3}  \frac{T_{k}^4}{9^k}\equiv -\frac{3p}{4}+\frac{3p^2}{4}\left( \frac{q_p(3)}{4}-1\right) \pmod{p^3},
\end{align*}
where $T_k$ is the central trinomial coefficient and $q_p(a)$ is the Fermat quotient.

\end{abstract}

\thanks{The first author is the corresponding author}

\maketitle

\section{Introduction}
\setcounter{lemma}{0} \setcounter{theorem}{0}
\setcounter{equation}{0}\setcounter{proposition}{0}
\setcounter{remark}{0}\setcounter{conjecture}{0}

Given integers $b,c$ and a nonnegative integer $n$, the generalized central trinomial coefficient, denoted by $T_n(b,c)$, is defined as the coefficient of $x^n$ in the expansion of $(x^2+bx+c)^n$ (see \cite{noe2006,sun-sci-2014}). By the multinomial theorem, we have
\begin{equation}
T_n(b,c)=\sum_{k\geq 0}\binom{n}{k}\binom{n-k}{k}b^{n-2k}c^k=\sum_{k=0}^n\binom{n}{2k}\binom{2k}{k}b^{n-2k}c^k.
\end{equation}
The generalized central trinomial coefficients are closely related to several well-known combinatorial sequences. For instance, $T_n(2,1)$ is the $n$th central binomial coefficient (OEIS sequence A000984 \cite{oeis}), $T_n(1,1)$ is the $n$th central trinomial coefficient (OEIS sequence A002426 \cite{oeis}), and $T_n(3,2)$ is the $n$th central Delannoy number (OEIS sequence A001850 \cite{oeis}). Let $d:=b^2-4c\neq 0$. The generalized central trinomial coefficients are also related to the Legendre polynomials (see Gould \cite[(3.135)]{gould}), defined by
\[
P_n(x):=\sum_{k=0}^n \binom{n}{k}\binom{n+k}{k}\left(\frac{x-1}{2}\right)^k,
\]
via the identity (see \cite{noe2006,sun-cant-2014,sun-sci-2014})
\[
T_n(b,c)=\left(\sqrt{d}\right)^n P_n\!\left(\frac{b}{\sqrt{d}}\right).
\]

In recent years, the congruence properties of generalized central trinomial coefficients have been studied by several authors. In particular, Sun \cite{sun-cant-2014,sun-sci-2014} made significant progress in this direction by establishing various congruences involving these coefficients. For example, he showed that for every positive integer $n$, we have
\begin{align*}
d^{n-1}\sum_{k=0} ^{n-1}(2k+1)\frac{T_k(b,c)^2}{d^{k}}\equiv 0\pmod{n^2}.
\end{align*}
 Mao and Liu \cite{mao2023} derived several interesting congruences of the form
\[
\sum_{k=0}^{p-1} (2k+1)^{2a+1}\,\varepsilon^{k}\,
\frac{T_k(b,c)^2}{d^{k}},
\]
for $a=0,1,2$ and $\varepsilon \in \{-1,1\}$, where $p>3$ is a prime and $\gcd(p,d)=1$. For instance, they established that

\begin{equation}\label{maoluicongruence}
\sum_{k=0}^{p-1}\left( 2k+1\right) ^{3}\frac{T_{k}\left( b,c\right) ^2}{d^{k}}\equiv 
\begin{cases}
-p^2 \pmod{p^3}, & \text{if } p\mid c,\\
-p^2-\frac{2}{3}p^2\left( \frac{d}{p}\right) +p^2\left(\frac{7d}{6c}+\frac{d^2}{3c^2}\right)\left(\left(\frac{d}{p}\right)-1\right)\pmod{p^3}, & \text{ if } p\nmid c,
\end{cases}
\end{equation}
where $(\div)$ denotes the Legendre symbols, given as 
\begin{align*}
\left(\frac{n}{p}\right):=\begin{cases}0, &\text{if $p$ divides $n$},\\
1, &\text{if $n$ is a quadratic residue modulo $p$}, \\
-1, &\text{ if $n$ is a quadratic nonresidue modulo $p$}. \end{cases}
\end{align*}
Moreover, Wang and Chen \cite{wang} demonstrated that; for any integer $m$, satisfies the equation $(m-d)^2=16mc$ with $p \nmid md(m^2-d^2)$, we have
\begin{align*}
\sum_{k=0}^{p-1}\frac{kT_{k}\left( b,c\right) ^2}{m^{k}}\equiv\ & \frac{3d-m}{4(m-k)}\left(\frac{-1}{p}\right)-\frac{pd}{2(m-d)}\left(\frac{dm}{p}\right)+\frac{p(3d^2-dm)}{4(m-d)^2}\left(\frac{-1}{p}\right)\\
&\times\left( S_{p-1}\left(\frac{m+d}{4d}\right)-S_{p-1}\left(\frac{m+d}{4m}\right)+q_p(m)-q_p(d)\right)\pmod{p^2},
\end{align*}
where $q(a):=(a^{p-1}-1)/p$ is the Fermat quotient ($a$ is a $p$-adic unit), and
\begin{align*}
S_{p-1}\left(x\right):=\sum_{k=1}^{p-1}\binom{2k}{k}\frac{x^k}{k}.
\end{align*}
For additional congruences concerning generalized central trinomial coefficients, we refer readers to \cite{chen2022,wang2021,guo}.

Motivated by the above results, in this paper we mainly establish some congruences involving the fourth powers of generalized central trinomial coefficients. Our first result is stated as follows.
\begin{theorem}\label{theorem1}Let $p \ge 5$ be a prime and let $b,c$ be integers. Set $d:=b^2-4c$ and suppose that $p\nmid d$. We distinguish the following cases:
\begin{itemize}
\item[(i)] If $p \mid c$, then
\begin{equation}\label{congruence1+}
\sum_{k=0}^{p-1}\left( 2k+1\right) \frac{T_k\left( b,c\right) ^4}{d^{2k}}\equiv p^2-p^2\frac{2c}{d}\pmod{p^4}.
\end{equation}
\item[(ii)] If $p\mid b$, then
\begin{align}\label{congruence1++}
\sum_{k=0}^{p-1}\left( 2k+1\right) \frac{T_k\left( b,c\right) ^4}{d^{2k}}\equiv p-p\frac{b^2}{4c}\pmod{p^4}.
\end{align}
\item[(iii)] If $p\nmid c$ and $p\nmid b$, then
\begin{align}
&\sum_{k=0}^{p-1}\left( 2k+1\right) \frac{T_k\left( b,c\right) ^4}{d^{2k}}\equiv   p+p^2\frac{b^2}{4c}\left( 2q_p(b)-q_p(d)\right)\nonumber\\ &+p^3\left(S_{p-1}^{(2)}\left( \frac{b^2}{d}\right)+\frac{d}{4c}\pounds_2\left( \frac{b^2}{d}\right)  + \frac{b^2}{4c}\left(q_p(b)-q_p(d)\right)^2 \right) \pmod{p^4}\label{congruence1},
\end{align} 
where 
\begin{align*}
S^{(2)}_{p-1}(x) :=\frac{2}{p}\left( \frac{2-(1-\sqrt{x})^p-(1+\sqrt{x})^p}{p}-\sum_{i=1}^{p-1}\frac{(1-\sqrt{x})^i+(1+\sqrt{x})^i}{i}\right)
\end{align*}
and 
\begin{align*}
\pounds_2\left( x\right):=\sum_{k=1}^{p-1}\frac{x^k}{k^2},
\end{align*}
in which $x$ is a $p$-adic number.

\end{itemize}

\end{theorem}
\begin{corollary}Let $p>3$ be a prime and $x\in\mathbb{Z}$. If $p\nmid x(x+1)(2x+1)$, we have
\begin{align}\label{congruence-legendre}
\sum_{k=0}^{p-1}\left( 2k+1\right)P_k\left(2x+1\right) ^4\equiv p+p^2\frac{(2x+1)^2q_p(2x+1)}{2x(x+1)} \pmod{p^3}.
\end{align}
\end{corollary}
\begin{proof}
Set $b=2x+1$ and $c=x^2+x$ for $x$ be an integer. One can observe that $T_n(2x+1,x^2+x)=P_n(2x+1)$. Obviously, \eqref{congruence-legendre} follows from \eqref{congruence1} for $b=2x+1$ and $c=x^2+x$.
\end{proof}
\begin{remark}
The congruence \eqref{congruence-legendre} was proved modulo $p^2$
  by Wang and Xia \cite[(1.4)]{wang2021}; it also appears in Guo’s work \cite[Theorem 1.3]{guo}.
\end{remark}
Inspired by \eqref{maoluicongruence}, we establish the following congruence.
\begin{theorem}\label{theorem2}
Let $p \ge 5$ be a prime and let $b,c$ be integers. Set $d:=b^2-4c$ and suppose that $p\nmid d$. We distinguish the following cases:
\begin{itemize}
\item[(i)]If $p\mid c$, then
\begin{align}\label{theorem21}
\sum_{k=0}^{p-1}\left( 2k+1\right)^3 \frac{T_k\left( b,c\right) ^4}{d^{2k}}\equiv -p^2 \pmod{p^3}.
\end{align}
\item[(ii)]If $p\mid b$, then
\begin{align}\label{theorem22}
\sum_{k=0}^{p-1}\left( 2k+1\right)^3 \frac{T_k\left( b,c\right) ^4}{d^{2k}}\equiv -p\pmod{p^3}.
\end{align}
\item[(iii)]If $p\nmid c$ and $p\nmid b$, then
\begin{equation}\label{theorem23}
\sum_{k=0}^{p-1}\left( 2k+1\right)^3 \frac{T_k\left( b,c\right) ^4}{d^{2k}}\equiv  p\frac{d}{4c} -p^2\frac{d}{4c}\left(\frac{b^2}{4c}\left( q_p(d)-2q_p(b)\right) +1+\frac{12c}{2d}\right) \pmod{p^3}.
\end{equation}
\end{itemize}

\end{theorem}
\begin{corollary}Let $p>3$ be a prime and $x\in\mathbb{Z}$. If $p\nmid x(x+1)(2x+1)$, we have
\begin{align*}
&\sum_{k=0}^{p-1}\left( 2k+1\right)^3P_k\left(2x+1\right) ^4\\& \equiv \frac{p}{4x(x+1)} -\frac{p^2}{4x(x+1)}\left(-\frac{(2x+1)^2q_p(2x+1)}{2x(x+1)}+6(x^2+x)+1\right) \pmod{p^3}.
\end{align*}
\end{corollary}
Finally, we have.
\begin{theorem}\label{theorem3}Let $p\geq 5$ be a prime. Let $b,c$ be integers. Set $d:=b^2-4c$ and suppose that $p\nmid d$. Then
\begin{equation}\label{theorem3-cg-1}
\sum_{k=0}^{p-1}\left( 2k+1\right)\left( -1\right) ^{k} \frac{T_k\left( b,c\right) ^4}{d^{2k}}\equiv p\sum_{k=0}^{p-1}\binom{2k}{k}^2\frac{(-c)^kb^{2k}}{d^{2k}}\pmod{p^3}
\end{equation} 
and 
\begin{equation}\label{theorem3-cg-2}
\sum_{k=0}^{p-1}\left( 2k+1\right)^3\left( -1\right) ^{k} \frac{T_k\left( b,c\right) ^4}{d^{2k}}\equiv -p\sum_{k=0}^{p-1}\binom{2k}{k}^2\frac{(-c)^kb^{2(k-1)}\left(3 b^2 + 4 b^2k+ 8 c k\right) }{d^{2k}}\pmod{p^3}.
\end{equation} 
\end{theorem}
\begin{remark}
Guo \cite[Theorem 1.4]{guo}  proved that \eqref{theorem3-cg-1} holds modulo $p^2$ for $b=2x+1$ and $c=x^2+x$. This congruence was later extended to modulo $p^3$ by Wang and Xia \cite[Theorem 1.2]{wang2021}.
\end{remark}

The paper is organized as follows. In Section \ref{section2}, we collect several binomial identities and useful congruences required in the proofs of the main theorems. These are obtained with the aid of the symbolic summation package \texttt{Sigma} \cite{sigma2007} and the Zeilberger's algorithm \cite{zeil} as implemented in \texttt{Mathematica}. Finally, Section \ref{section3} is devoted to the proofs of the theorems stated above.

\medskip

\section{Auxiliary Results}
\setcounter{lemma}{0} \setcounter{theorem}{0}
\setcounter{equation}{0}\setcounter{proposition}{0}
\setcounter{remark}{0}\setcounter{conjecture}{0}

\label{section2}
First, we give some notations. 
\begin{itemize}
\item Let $n$ and $m$ are positive integers. The $n$th generalized harmonic number of order $m$ is given by
\begin{align*}
H_n^{(m)}:=\sum_{k=1}^n\frac{1}{k^m} \text{,}\quad H_0^{(m)}:=0.
\end{align*}
In particular, for $m=1$ we get the harmonic numbers $H_n^{(1)}=H_n$.
\item Let $p$ be a prime. We denote by $\mathbb{Z}_p$ the set of rational $p$-integers, i.e, those rational numbers whose denominator is not divisible by $p$.
\item Let $p$ be a prime, $x\in  \mathbb{Z}_p $, and $m$ a positive integer. Define the finite polylogarithms (cf. \cite{matt2013}) as follows:
\begin{align*}
\pounds_m(x):=\sum_{k=1}^{p-1}\frac{x^k}{k^m}.
\end{align*}

\item Let $p$ be a prime and $x,y\in  \mathbb{Z}_p $. Define
\begin{align*}
Q_p(x,y):=\frac{(y-x)^p+x^p-y^p}{p}=-yq_p(y)+xq_p(x)+(y-x)q_p(x-y).
\end{align*}
In view of Fermat’s Little Theorem, we have $Q_p(x,y) \in\mathbb{Z}_p$.
\item Let $n,i,j$ be nonnegative integers, define
\begin{align*}
\vartheta_{n,i,j}:=(n+i)(n+j)\binom{2i}{i}\binom{2j}{j}\binom{n+i-1}{2i}\binom{n+j-1}{2j}.
\end{align*}
\item Throughout the rest of the paper we consider $p>3$ as a prime number.
\end{itemize}

Now, we give the following binomial identity which due to Sun.
\begin{lemma}[{Sun \cite[Lemma 4.1]{sun-sci-2014}}]
Let $n,b,c$ be integers with $n\geq0$. We have
\begin{align}\label{sunidentity}
T_n\left(b,c\right)^2=\sum_{k=0}^{n} \binom{2k}{k}^2\binom{n+k}{2k}d^{n-k}c^k,
\end{align} 
 where $d:=b^2-4c\neq0$. 
\end{lemma} 
 Next, we establish some binomial identities by using  \texttt{Sigma} package.
\begin{lemma}Let $n$ be a nonnegative integer and $x$ be an indeterminate. Then 
\begin{align}\label{identity1}
\sum_{k=1}^n\binom{n}{k}\frac{(-1)^k}{k}\sum_{i=1}^{k}\frac{x^i}{i}=-H_n^{(2)}+\sum_{k=1}^n\frac{(1-x)^k}{k^2},
\end{align}
\begin{align}\label{identity3}
\sum_{i=1}^{n}\binom{n}{k}\frac{(-1)^{k}(5k+6)}{(k+4)(k+3)k}= -\frac{n \left(n^3+10 n^2+35 n+2\right)}{8 (n+1) (n+2) (n+3) (n+4)}-\frac{1}{2}H_{n},
\end{align}
\begin{align}\label{identity2+}
\sum_{k=0}^{n}\binom{2k}{k}\binom{2(n-k)}{n-k}\left((2+3n+4k(n-k)\right)=\frac{(1+n)(4+n)4^n}{2}.
\end{align}
\end{lemma} 
\begin{proof}
For brevity, we present only the proof of identity \eqref{identity1}. The same method applies to establish identities \eqref{identity3} and \eqref{identity2+}, as well as any further identities of this type. We begin by inserting our formula as follows:
\begin{align*}
\mathtt{In [1]}&\mathtt{:=mysum = SigmaSum}[\mathtt{SigmaBinomial[n, k]\times SigmaPower[-1, k]/k\times }\\ & \quad \quad\mathtt{SigmaSum[ SigmaPower[x,i]/i, \left\lbrace i, 1, k\right\rbrace], \left\lbrace k, 1, n\right\rbrace ]} \\
\mathtt{Out [1]} &=\mathtt{\sum_{k=1}^n\binom{n}{k}\frac{(-1)^k}{k}\sum_{i=1}^{k}\frac{x^i}{i}.}
\end{align*}
Next, we compute the recurrence relation of right hand side of \eqref{identity1} by using $\mathtt{GenerateRecurrence}$ command as follows:
\begin{align*}
\mathtt{In [2]}&\mathtt{:=rec=GenerateRecurrence[mySum, n][[1]]} \\
\mathtt{Out [2]}& \; \mathtt{=-(1+n)^2(-1+x)SUM[n]+(-5-6n-2n^2+x+2nx+n^2x)SUM[1+n]}\\ 
& \quad \quad\mathtt{+(2+n)^2SUM[2+n]==-x}.
\end{align*}
And then we solve the above recurrence as follows:
\begin{align*}
\mathtt{In [3]}&\mathtt{:=recSol = SolveRecurrence[rec, SUM[n]]} \\
\mathtt{Out [3]}&\; \mathtt{=  \left\{\{0,1\},\left\{0,-\left(\underset{\iota=1}{\overset{n}{\sum}}\frac{(1-x)^{\iota}}{\iota^2}\right)\right\},\left\{1,-\left(\underset{\iota=1}{\overset{n}{\sum}}\frac{1}{\iota^2}\right)\right\}\right\}}.
\end{align*}
Finally, to find a linear combination of the solutions for our sum $\mathtt{In[1]}$, we use the following step:
\begin{align*}
\mathtt{In [4]}&\mathtt{:=FindLinearCombination[recSol, mySum, n, 2]}\\
\mathtt{Out [4]}&\; \mathtt{= -\left(\underset{\iota=1}{\overset{n}{\sum}}\frac{1}{\iota^2}\right)+\underset{\iota=1}{\overset{n}{\sum}}\frac{(1-x)^{\iota}}{\iota^2}}.
\end{align*}
The proof is done.
\end{proof} 
Furthermore, we have the following binomial identities.
\begin{lemma}Let $n$ be a positive integer and $i,j$ be nonnegative integers. Then
\begin{equation}\label{1.1}
\sum_{k=\max\left( i,j\right) }^{n-1}\left( 2k+1\right) \binom{k+i}{2i}\binom{k+j}{2j}=\binom{n+i-1}{2i}\binom{n+j-1}{2j}\frac{\left( i+n\right) \left( j+n\right) }{1+i+j},
\end{equation} 
\begin{equation}\label{1.2}
\sum_{k=\max\left( i,j\right) }^{n-1}\left( 2k+1\right)^3 \binom{k+i}{2i}\binom{k+j}{2j}
=-\binom{n+i-1}{2i}\binom{n+j-1}{2j}\frac{\left( i+n\right) \left( j+n\right)\mu_{n,i,j} }{(1+i+j)(2+i+j)},
\end{equation} 
where $\mu_{n,i,j}:=2+3i+3j+4ij-4n^2-4in^2-4jn^2$.
 \end{lemma}
\begin{proof}
Again, by \texttt{Sigma} package, we can easily derive the identities \eqref{1.1} and \eqref{1.2} by applying the function $\mathtt{SigmaReduce}$. Now, by straight-forward induction on $n$, for $0\leq i,j \leq n-1$, we can prove \eqref{1.1} and \eqref{1.2}.
\end{proof} 
Now, we have the following binomial identity.
\begin{lemma}Let $n$ be a nonnegative integer. Then
\begin{equation}\label{catalan-new}
\sum_{k=0}^n\binom{n}{k}\binom{2k}{k}\binom{2n-2k}{n-k}(n-k)k=-\frac{1}{2}\sum_{k=0}^{n}\binom{2k}{k}^2\binom{k}{n-k}(n-k)(-4)^{n-k}.
\end{equation}
\end{lemma}
\begin{proof}
Let $r_n$ be the right-hand side of \eqref{catalan-new} for each nonnegative integer $n$. Using the symbolic summation package \texttt{Sigma}, we find that $(r_n)_n$ satisfies the following recurrence relation:
\begin{align*}
32 (1 + n)^2r_n - 4 (1 + 5 n + 3 n^2)r_{n+1} + 
 n (1 + n)r_{n+2}=0.
\end{align*}
In addition, we found that the left hand side of \eqref{catalan-new} verifies the same recurrence relation. Note that both sides of \eqref{catalan-new} coincide for $n = 0, 1$. Thus \eqref{catalan-new} holds for any nonnegative integer $n$.
\end{proof}
In the following lemmas, we establish several auxiliary congruences needed for the proofs of the theorems stated above.
\begin{lemma}\label{lemma2}Let $0\leq i,j\leq p-1$ be integers. Then 
\begin{equation}
\vartheta_{p,i,j}\equiv \left( -1\right) ^{i+j}p^2\left( 1-p^2\left( H_{i}^{(2)}+H_{j}^{(2)}\right)\right)  \pmod{p^6}.
\end{equation}
\end{lemma}
\begin{proof}
Note that 
\begin{align}\label{bc11}
(p+i)\binom{2i}{i}\binom{p+i-1}{2i}=p(-1)^i\prod_{k=1}^{i}\left( 1-\frac{p^2}{k^2}\right) .
\end{align}
It follows that
\begin{align}\label{varformula}
\vartheta_{p,i,j}=p^2(-1)^{i+j}\prod_{k=1}^{i}\prod_{l=1}^{j}\left( 1-\frac{p^2}{k^2}\right)\left( 1-\frac{p^2}{l^2}\right) .
\end{align}
Hence,
\begin{align*}
\vartheta_{p,i,j}\equiv p^2(-1)^{i+j}\prod_{k=1}^{i}\prod_{l=1}^{j}\left( 1-p^2\left( \frac{1}{k^2}+\frac{1}{l^2}\right)\right) \pmod{p^6}.
\end{align*}
Set $$f(x,i,j):=\prod_{k=1}^{i}\prod_{l=1}^{j}\left( 1-x\left( \frac{1}{k^2}+\frac{1}{l^2}\right)\right).$$
We have $f(0,i,j) =1$. Also,
\begin{align*}
\log(f(x,i,j))=\sum_{k=1}^i\sum_{l=1}^j\log\left(1-x\left(\frac{1}{k^2}+\frac{1}{l^2}\right)\right).
\end{align*}
Differentiating both sides of the above equation with respect to $x$ and then setting $x=0$, we obtain
\begin{align*}
f'(0,i,j)=-H_i^{(2)}-H_j^{(2)}.
\end{align*}
Applying Taylor series expansion of order 2, we get
\begin{align*}
f(x,i,j)=1-x\left( H_i^{(2)}+H_j^{(2)}\right) \mathcal{O}(x^2) .
\end{align*}
Finally, setting $ x = p^2 $ yields the desired result.
\end{proof}

\begin{lemma}Let $x,y$ are integers, in which $p\nmid xy$. Then
\begin{align}\label{cg2.9}
\pounds_{1}\left(\frac{x}{y}\right)\equiv -\frac{1}{y} Q_p(x,y)-p\left(\pounds_2\left(1-\frac{x}{y}\right)-\frac{q_p(y)Q_p(x,y)}{y}\right)\pmod{p^2}.
\end{align}
\end{lemma}

\begin{proof}Let $x,y$ are integers with $y\neq0$. One can observe that
\begin{align*}
Q_p\left(1-\frac{x}{y},1\right)=\frac{1}{y^p}Q_p(x,y).
\end{align*}
For any $p$-adic integer $a$, Granville \cite[Eq. (6)]{gra2004}, proved that
\begin{align*}
\pounds_1(1-a)\equiv -Q_p(a,1)-p\pounds_2(a) \pmod{p^2}.
\end{align*}
Substituting $ x/y = a $ into the above congruence and using the fact that $ y^{p-1} = 1 + p\,q_p(y) $ for $ p \nmid y $, we obtain \eqref{cg2.9}.
\end{proof}

\begin{lemma}
If $x$ is a p-adic integer, then
\begin{equation}\label{cgh2}
\sum_{k=1}^{(p-1)/2}\frac{x^k}{k^2}\equiv S^{(2)}_{p-1}(x) \pmod{p}.
\end{equation}
\end{lemma}

\begin{proof}
Observe that 
\begin{align}\label{formulah2}
\sum_{k=1}^{(p-1)/2}\frac{x^k}{k^2}=2\left(\sum_{k=1}^{p-1}\frac{\left( -\sqrt{x}\right) ^k}{k^2}+\sum_{k=1}^{p-1}\frac{\left( \sqrt{x}\right) ^k}{k^2}\right).
\end{align}
Sun \cite[Lemma 4.1]{hsun2008} obtains the following congruences: for an odd prime $p$, in $\mathbb{Z}_p\left[ x\right] $ we have
\begin{align}\label{hsuncg}
\sum_{k=1}^{p-1}\frac{x^k}{k^2}\equiv \frac{1}{p}\left( \frac{1+(x-1)^p-x^p}{p}-\sum_{i=1}^{p-1}\frac{(1-x)^i}{i}-H_{p-1}\right) \pmod{p}.
\end{align}
The Wolstenholme theorem states that (see \cite{wols});
\begin{align}\label{wlscg}
H_{p-1}\equiv 0\pmod{p^2}.
\end{align} 
Substituting \eqref{hsuncg} into \eqref{formulah2} and using \eqref{wlscg} yields the desired result.
\end{proof}

\begin{lemma}\label{cgH(2)} If $0\leq k \leq p-1$, then
\begin{equation*}
H_{p-1-k}^{(2)}\equiv-H_{k}^{(2)} \pmod{p}.
\end{equation*}
\end{lemma}
\begin{proof}
For any $0\leq k \leq p-1$, we have
\begin{align*}
H_{p-1-k}^{(2)}=\sum_{i=k+1}^{p-1}\frac{1}{(p-i)^2}\equiv H_{p-1}^{(2)}-H_{k}^{(2)} \pmod{p}.
\end{align*}
Finally, setting $ x = 1 $ in \eqref{hsuncg} leads to $ H_{p-1}^{(2)} \equiv 0 \pmod{p} $. This completes the proof.
\end{proof}

\begin{lemma}\label{lemma7+} Let $k,i$ be nonnegative integers. \\
If $0\leq i<(p-1)/2$, then
\begin{align}\label{Lemmac13}
\binom{2(p-1)-2i}{p-1-i}\equiv -\frac{p}{\binom{2i}{i}(2i+1)}\pmod{p^2},
\end{align}
\begin{align}\label{Lemmac13+}
\binom{(p-1)-2i}{(p-1)/2-i}\equiv \binom{2i}{i}\frac{(-1)^{(p-1)/2}}{16^i}\pmod{p}.
\end{align}
If $0< k\leq \left( p-1\right) /2$, then
\begin{align}\label{Lemma7c14}
\binom{p-1+2k}{\left( p-1\right) /2+k}\equiv p\frac{4^{2k}\left(-1\right)^{\left( p-1\right) /2}}{2k\binom{2k}{k}}\pmod {p^2}.
\end{align}
\end{lemma}
\begin{proof}For $m$ be a positive integer and $i$ be a nonnegative integer, one can see that
\begin{align}\label{cbidentity}
\binom{2m-2i}{m-i}=\frac{\binom{2m}{m}\binom{m}{i}^2}{\binom{2i}{i}\binom{2m}{2i}}.
\end{align}
For $ i \in \{0, \ldots, (p-1)/2\} $ and $ k \in \{0, \ldots, p-1\} $, it is routine to verify that
\begin{align*}
\binom{(p-1)/2}{i} \equiv \binom{2i}{i}\frac{1}{(-4)^i}\pmod p \text{ and }\binom{p-1}{k}\equiv (-1)^k\pmod{p}.
\end{align*}
Also, we have
\begin{align*}
\binom{2(p-1)}{p-1}=\frac{(2p-2)(2p-3)\cdots (p+1)p}{(p-1)!}\equiv -p\pmod{p^2}.
\end{align*}
Now, \eqref{Lemmac13} and \eqref{Lemmac13+} follow by substituting $m=p-1$ (respectively $m=(p-1)/2$) into \eqref{cbidentity} and using the above congruences.

On other hand side, we have
\begin{align*}
\binom{p-1+2k}{\left( p-1\right) /2+k}&=\binom{p-1}{\left( p-1\right) /2}\frac{\prod_{i=1}^{2k}\left( p-1+i\right) }{\prod_{i=1}^{k}\left( \left( p-1\right) /2+i\right)^{2} }\\
&=p\binom{p-1}{\left( p-1\right) /2}\frac{2^{2k}\prod_{i=1}^{2k-1}\left( p+i\right) }{\prod_{i=1}^{k}\left(p+2i-1\right)^{2} }.
\end{align*}
We readily obtain
\begin{align*}
\frac{\prod_{i=1}^{2k-1}\left( p+i\right) }{\prod_{i=1}^{k}\left(p+2i-1\right)^{2} }\equiv\frac{2^{2k}}{2k\binom{2k}{k}}\pmod p.
\end{align*}
Finally, combining Morley’s congruence \cite{morley}
\begin{align}\label{morleycg}
\binom{p-1}{(p-1)/2} \equiv (-1)^{(p-1)/2} 4^{p-1} \pmod{p^3},
\end{align}
with Fermat’s Little Theorem, we obtain \eqref{Lemma7c14}.
\end{proof}

\begin{lemma}Let $0<k\leq(p-1)/2$ be an integer. Then
\begin{equation}\label{lemma71}
\sum_{i=k}^{(p-1)/2} \frac{\binom{2(i-k)}{i-k}}{\binom{2i}{i}i}\equiv -\frac{H_{(p-1)/2}}{4^{k}}-\frac{1}{4^k}\sum_{i=1}^{k-1}\binom{2i}{i}\frac{1}{i4^{i}}\pmod{p},
\end{equation}
\begin{equation}\label{lemma72}
\sum_{i=0}^{(p-3)/2}\binom{2(k+i)}{k+i}\frac{1}{(2i+1)\binom{2i}{i}}\equiv\frac{4^k}{2}H_{(p-1)/2}+\frac{1}{2}\sum_{i=1}^{k-1}\binom{2i}{i}\frac{4^{k-i}}{i}\pmod{p}.
\end{equation}

\end{lemma}
\begin{proof}We prove only \eqref{lemma71}, since \eqref{lemma72} can be established by a similar argument. Let $n,k$ are positive integers, with $0<k\leq n$. Define
\begin{align*}
z_k=z(k,n):=\sum_{i=k}^{n} \frac{\binom{2(i-k)}{i-k}}{\binom{2i}{i}i}.
\end{align*}
Applying Zeilberger’s algorithm, using the \texttt{Fast Zeilberger} package \cite{paule} implemented in \texttt{Mathematica}, we find that the sequence $ (z_k)_{k}$ satisfies the following recurrence relation:
\begin{align*}
z_k-4z_{k+1}=\frac{\binom{2n-2k}{n-k}}{k\binom{2n}{n}}.
\end{align*}
Set $n=(p-1)/2$. Using \eqref{Lemmac13+}, Morley's congruence \eqref{morleycg} and Fermat's Little Theorem, we obtain
\begin{align*}
z_k-4z_{k+1}\equiv \binom{2k}{k}\frac{1}{k16^k}\pmod{p}.
\end{align*}
Finally, solving the above recurrence relation modulo $p$, for $k\in \left\lbrace 1,2,\ldots,(p-1)/2\right\rbrace $, we arrive at \eqref{lemma71}.
\end{proof}

\begin{lemma}\label{lemma11} Let $b,c$ be integers. Set $d:=b^2-4c$ and suppose that $p\nmid d$. Then
\begin{align*}
\sum_{k=1}^{(p-1)/2}\frac{(-4c)^k}{d^kk}\sum_{i=1}^{k-1}\frac{\binom{2i}{i}}{4^{i}i}\equiv-H_{(p-1)/2}\sum_{k=1}^{(p-1)/2}\frac{(-4c)^k}{d^kk}-\sum_{k=1}^{(p-1)/2}\frac{(d+4c)^k}{d^kk^2}\pmod{p}.
\end{align*}
\end{lemma}
\begin{proof}
It is easy to show that 
\begin{align*}
\sum_{k=1}^{(p-1)/2}\frac{(-4c)^k}{d^kk}\sum_{i=1}^{k-1}\frac{\binom{2i}{i}}{4^{i}i}=\sum_{i=1}^{(p-1)/2}\frac{\binom{2i}{i}}{4^{i}i}\sum_{k=1}^{(p-1)/2}\frac{(-4c)^k}{d^kk}-\sum_{i=1}^{(p-1)/2}\frac{\binom{2i}{i}}{4^{i}i}\sum_{k=1}^{i}\frac{(-4c)^k}{d^kk}.
\end{align*}
We remember that $\binom{2k}{k}\equiv \binom{(p-1)/2}{k}(-4)^k\pmod{p}$ for $k\in \left\lbrace 0,1,\ldots,(p-1)/2\right\rbrace $. Hence, using \eqref{identity1} and the following binomial identity (see Gould \cite[(1.45)]{gould}):
\begin{align*}
\sum_{i=1}^{n}\binom{n}{i}\frac{(-1)^i}{i}=-H_n,
\end{align*}
we get
\begin{align*}
\sum_{k=1}^{(p-1)/2}\frac{(-4c)^k}{d^kk}\sum_{i=1}^{k-1}\frac{\binom{2i}{i}}{4^{i}i}\equiv-H_{(p-1)/2}\sum_{k=1}^{(p-1)/2}\frac{(-4c)^k}{d^kk}+H_{(p-1)/2}^{(2)}-\sum_{k=1}^{(p-1)/2}\frac{(d+4c)^k}{d^kk^2}\pmod{p}.
\end{align*}
Substituting $ x = 1 $ into \eqref{cgh2} yields $ H_{(p-1)/2}^{(2)} \equiv 0 \pmod{p} $. This completes the proof.
\end{proof}

\begin{lemma} Let $n,i,j$ be nonnegative integers with $n>0$. Define
\begin{align*}
z(n,i,j,a):=\binom{2i}{i}\binom{2j}{j}\sum_{k=\max(i,j)}^{n-1}\binom{k+i}{2i}\binom{k+j}{2j}(2k+1)^{a}(-1)^k.
\end{align*}
If $0\leq j<p-1$ and $0\leq i\leq p-1$, then
\begin{align}\label{theorem3cg1}
z(p,i,j,1)\equiv p(-1)^{i+j}\binom{i+j}{j}\pmod{p^3},
\end{align}
\begin{align}\label{theorem3cg1+}
z(p,i,j,3)\equiv p(-1)^{i+j}\binom{i+j}{j}\left( 3+4i+4j+4ij\right) \pmod{p^3}.
\end{align}
\end{lemma}
\begin{proof}
Let $ n, i, j $ be nonnegative integers with $ n > 0 $, and set $ z_j = z(n,i,j) $. Applying Zeilberger’s algorithm, we find that the sequence $ (z_j)_{j } $ satisfies the following recurrence relation:
\begin{align*}
-(1 + i + j)z_j-(1 +j)z_{1 + j} =\binom{2i}{i}\binom{2j}{j}\binom{n-1+i}{2i}\binom{n-1+j}{2j}\frac{(-1)^nn (i + n) (j + n)}{1+j}.
\end{align*}
Replacing $n$ by $p$ into the above equation, and using Lemma \ref{lemma2} modulo $p^3$, for $0\leq j<p-1$ and $0\leq i\leq p-1$, we obtain
\begin{align*}
(1+i+j)z_j+(1+j)z_{1+j} \equiv 0\pmod{p^3}.
\end{align*}
Which implies that 
\begin{align*}
z_j\equiv (-1)^j\binom{i+j}{j}z_0\pmod{p^3}.
\end{align*}
Note that
\begin{align*}
z_0=\binom{2i}{i}\sum_{k=i}^{p-1}\binom{k+i}{2i}(2k+1)(-1)^k.
\end{align*}
For $n$ be a positive integer and $0\leq j\leq n-1 $ be an integer, we have (see \cite[(3.1)]{mao2023}),
\begin{align*}
\sum_{k=i}^{n-1}\binom{k+i}{2i}(2k+1)(-1)^k=(i+n)\binom{n-1+i}{2i}.
\end{align*}
Substituting $ p = n $ into the above binomial identity and reducing modulo $ p^3 $ via \eqref{bc11}, we obtain \eqref{theorem3cg1}.  

Similarly, applying the same procedure and using the binomial identity \cite[(3.2)]{mao2023}, we deduce \eqref{theorem3cg1+}.
\end{proof}

\medskip

\section{Proofs of theorems}
\setcounter{lemma}{0} \setcounter{theorem}{0}
\setcounter{equation}{0}\setcounter{proposition}{0}
\setcounter{remark}{0}\setcounter{conjecture}{0}

\label{section3}
\begin{proof}[Proof of Theorem \ref{theorem1}] Applying identity \eqref{sunidentity}, we obtain
\begin{align*}
\sum_{k=0}^{p-1}\left( 2k+1\right) \frac{T_k\left( b,c\right) ^4}{d^{2k}}&=\sum_{k=0}^{p-1}\frac{\left( 2k+1\right) }{d^{2k}}\sum_{i=0}^{k} \sum_{j=0}^{k} \binom{2i}{i}^2\binom{2j}{j}^2\binom{k+i}{2i}\binom{k+j}{2j}d^{2k-i-j}c^{j+i}\\
&=\sum_{i=0}^{p-1} \sum_{j=0}^{p-1} \binom{2i}{i}^2\binom{2j}{j}^2\frac{c^{j+i}}{d^{i+j}}\sum_{k=\max\left( i,j\right) }^{p-1}\left( 2k+1\right)\binom{k+i}{2i}\binom{k+j}{2j}.
\end{align*}
Using identity \eqref{1.1} with $n=p$, we get
\begin{align*}
\sum_{k=0}^{p-1}\left( 2k+1\right) \frac{T_k\left( b,c\right) ^4}{d^{2k}}=\sum_{i=0}^{p-1} \sum_{j=0}^{p-1} \binom{2i}{i}\binom{2j}{j}\frac{\vartheta_{p,i,j} }{1+i+j}\frac{c^{j+i}}{d^{i+j}}.
\end{align*}
Set
\begin{align*}
\Sigma_1:=\sum_{i=0}^{p-2} \sum_{j=0}^{p-2-i} \binom{2i}{i}\binom{2j}{j}\frac{\vartheta_{p,i,j} }{1+i+j}\frac{c^{j+i}}{d^{i+j}},\\
\Sigma_2:=\sum_{i=1}^{p-1} \sum_{j=p-i}^{p-1} \binom{2i}{i}\binom{2j}{j}\frac{\vartheta_{p,i,j} }{1+i+j}\frac{c^{j+i}}{d^{i+j}}
\end{align*}
and
\begin{align*}
\Sigma_3:=\frac{c^{p-1}}{d^{p-1}}\frac{1}{p}\sum_{i=0}^{p-1}\binom{2i}{i}\binom{2(p-1)-2i}{p-1-i}\vartheta_{p,i,p-1-i}.
\end{align*}
It is easy to see that
\begin{align*}
\sum_{k=0}^{p-1}\left( 2k+1\right) \frac{T_k\left( b,c\right) ^4}{d^{2k}}=\Sigma_1+\Sigma_2+\Sigma_3.
\end{align*}

We first evaluate $\Sigma_1$ modulo $p^4$. Applying Lemma \ref{lemma2} modulo $p^4$, we get
\begin{align*}
\Sigma_1\equiv\ & p^2\sum_{i=0}^{p-2} \sum_{j=0}^{p-2-i} \binom{2i}{i}\binom{2j}{j}\frac{c^{j+i}}{d^{i+j}}\frac{\left( -1\right)^{i+j} }{1+i+j}\pmod{p^4}\\
=\ & p^2\sum_{i=0}^{p-2} \sum_{k=i}^{p-2} \binom{2i}{i}\binom{2(k-i)}{k-i}\frac{c^{k}}{d^{k}}\frac{\left( -1\right)^{k} }{1+k}\pmod{p^4}\\
=\ & p^2 \sum_{k=0}^{p-2}\frac{c^{k}}{d^{k}}\frac{\left( -1\right)^{k} }{1+k} \sum_{i=0}^{k}\binom{2i}{i}\binom{2(k-i)}{k-i}\pmod{p^4}.
\end{align*}
Using the following binomial identity (see Gould \cite[(3.90)]{gould}):
\begin{align}\label{identity2}
\sum_{i=0}^k\binom{2k-2i}{k-i}\binom{2i}{i}=4^k,
\end{align} 
we arrive at
\begin{align}
\Sigma_1&\equiv p^2 \sum_{k=0}^{p-2}\frac{(-4c)^{k}}{d^{k}(k+1)}\pmod{p^4}.
\end{align}
We distinguish three cases. If $p\mid c$, we deduce that
\begin{align}\label{sigma1cg+}
\Sigma_1\equiv p^2-p^2\frac{2c}{d}\pmod{p^4}.
\end{align}
If $p\mid b$, we obtain
\begin{align}
\Sigma_1 &\equiv p^2 \sum_{k=0}^{p-2}\frac{1}{(k+1)}\pmod{p^4}\nonumber\\
&\equiv 0\pmod{p^4},\label{sigma1cg++}
\end{align}
where in last step we use \eqref{wlscg}. \\
Now, set $p\nmid c$ and $p\nmid b$. We have
\begin{align*}
\Sigma_1\equiv-\frac{d}{4c}p^2\sum_{k=1}^{p-1}\frac{(-4c)^{k}}{d^{k}k}\pmod{p^4},
\end{align*}
and by \eqref{cg2.9}, we find that
\begin{align*}
\Sigma_1\equiv p^2\frac{Q_p(-4c,d)}{4c} +p^3\frac{1}{4c}\left(d\pounds_2\left(1+\frac{4c}{d}\right)-q_p(d)Q_p(-4c,d)\right)\pmod{p^4}.
\end{align*}
Note that
\begin{align*}
Q_p(-4c,d)&=-dq_p(d)-4cq_p(4c)+b^2q_p(b^2).
\end{align*}
Moreover, using the Fermat quotient, one easily verifies that $q_p(b^2)=2q_p(b)+pq_p(b)^2$. Hence,
\begin{align}
\Sigma_1\equiv p^2\left( \frac{-dq_p(d)+2b^2q_p(b)}{4c} -q_p(4c)\right)&+p^3\frac{1}{4c}\left(d\pounds_2\left(1+\frac{4c}{d}\right)+b^2q_p(b)^2 \right. \nonumber \\& \left. -q_p(d)\left( -dq_p(d)-4cq_p(4c)+2b^2q_p(b)\right) \right)\pmod{p^4}.\label{sigma1cg}
\end{align}

We next evaluate $ \Sigma_2 $ modulo $ p^4 $. Again applying Lemma \ref{lemma2}, we obtain
\begin{align*}
\Sigma_2\equiv\ & p^2\sum_{i=1}^{p-1} \sum_{j=p-i}^{p-1} \binom{2i}{i}\binom{2j}{j}\frac{c^{j+i}}{d^{i+j}(1+i+j)} \pmod{p^4}\\
=\ & p^2 \sum_{i=0}^{p-2} \sum_{k=p}^{p-1+i} \binom{2i}{i}\binom{2(k-i)}{k-i}\frac{c^{k}}{d^{k}}\frac{\left( -1\right)^{k} }{1+k}\pmod{p^4}\\
=\ & p^2 \sum_{i=1}^{p-1} \sum_{k=1}^{i} \binom{2i}{i}\binom{2(p-1)+2(k-i)}{p-1+k-i}\frac{c^{p-1+k}}{d^{p-1+k}}\frac{\left(-1\right)^{p-1+k} }{p+k}\pmod{p^4}.
\end{align*}
We distinguish three cases. If $p\mid c$, we easily get 
\begin{align}\label{sigma2cg+}
\Sigma_2\equiv 0\pmod{p^4}.
\end{align}
Suppose that $p\nmid c$. Since $\binom{2j}{j}\equiv 0\pmod{p}$ for $(p-1)/2<j<p$, and by the Fermat's Little Theorem, we arrive at
\begin{align*}
\Sigma_2&\equiv p^2 \sum_{i=1}^{p-1} \sum_{k=1}^{i} \binom{2i}{i}\binom{2(p-1)-2(i-k)}{p-1-(i-k)}\frac{c^{k}}{d^{k}}\frac{\left(-1\right)^{k} }{k}\pmod{p^4}\\
&=p^2\sum_{k=1}^{p-1} \frac{c^{k}}{d^{k}}\frac{\left(-1\right)^{k} }{k}\sum_{i=0}^{p-1-k} \binom{2i+2k}{i+k}\binom{2(p-1)-2i}{p-1-i}\pmod{p^4}\\
&\equiv p^2\sum_{k=1}^{(p-1)/2} \frac{c^{k}}{d^{k}}\frac{\left(-1\right)^{k} }{k}\sum_{i=0}^{p-1-k} \binom{2i+2k}{i+k}\binom{2(p-1)-2i}{p-1-i}\pmod{p^4}.
\end{align*}



In view of \eqref{Lemmac13} and \eqref{Lemma7c14}, we obtain
\begin{align*}
\Sigma_2
\equiv\ & \frac{p^3}{2}\sum_{k=1}^{(p-1)/2}\frac{4^{2k}(-c)^k}{d^kk}\sum_{i=0}^{(p-1)/2-k}\frac{\binom{2i}{i}}{(k+i)\binom{2(k+i)}{k+i}}\\ &-p^3\sum_{k=1}^{(p-1)/2}\frac{(-c)^k}{d^kk}\sum_{i=0}^{(p-3)/2}\binom{2(k+i)}{k+i}\frac{1}{(2i+1)\binom{2i}{i}}\pmod{p^4}.
\end{align*}
Applying \eqref{lemma71}, we find that
\begin{align*}
\frac{1}{2}\sum_{k=1}^{(p-1)/2}&\frac{4^{2k}(-c)^k}{d^kk}\sum_{i=0}^{(p-1)/2-k}\frac{\binom{2i}{i}}{(k+i)\binom{2(k+i)}{k+i}}\\\equiv& \frac{-H_{(p-1)/2}}{2}\sum_{k=1}^{(p-1)/2}\frac{(-4c)^k}{d^kk}-\frac{1}{2}\sum_{k=1}^{(p-1)/2}\frac{(-4c)^k}{d^kk}\sum_{i=1}^{k-1}\frac{\binom{2i}{i}}{i4^{i}}\pmod p.
\end{align*}
Also, by \eqref{lemma72}, we get
\begin{align*}
-\sum_{k=1}^{(p-1)/2}\frac{(-c)^k}{d^kk}&\sum_{i=0}^{(p-3)/2}\binom{2(k+i)}{k+i}\frac{1}{(2i+1)\binom{2i}{i}}\\\equiv&-\frac{H_{(p-1)/2}}{2}\sum_{k=1}^{(p-1)/2}\frac{(-4c)^k}{d^kk}-\frac{1}{2}\sum_{k=1}^{(p-1)/2}\frac{(-4c)^k}{d^kk}\sum_{i=1}^{k-1}\frac{\binom{2i}{i}}{4^{i}i}\pmod{p}.
\end{align*}
Therefore, combining the above results, and by Lemma \ref{lemma11}, we deduce that
\begin{align*}
\Sigma_2\equiv p^3\sum_{k=1}^{(p-1)/2}\frac{(d+4c)^k}{d^kk^2}\pmod{p^4}.
\end{align*}
Now, if $p\mid b$, we easily find
\begin{align}\label{sigma2cg++}
\Sigma_2\equiv 0\pmod{p^4}.
\end{align}
In otherwise, with \eqref{cgh2}, we conclude that
\begin{equation}\label{sigma2cg}
\Sigma_2\equiv p^3 S_{p-1}^{(2)}\left( \frac{d+4c}{d}\right)  \pmod{p^4}.
\end{equation}

Finally, we evaluate $ \Sigma_3 $ modulo $ p^4 $. In view of \eqref{varformula}, we show that
\begin{align}\label{sg3cg}
\frac{\vartheta_{p,i,p-1-i}}{p}
&\equiv p\left(1-p^2\left(H_{p-1-i}^{(2)}+H_{i}^{(2)}\right)\right)\pmod{p^4}.
\end{align} 
Using Lemma \ref{cgH(2)}, and by \eqref{sg3cg} with identity \eqref{identity2}, we obtain
\begin{align*}
\Sigma_3&\equiv p\frac{c^{p-1}}{d^{p-1}}\sum_{i=0}^{p-1}\binom{2i}{i}\binom{2(p-1)-2i}{p-1-i}\pmod{p^4}\\
&=p\frac{(4c)^{p-1}}{d^{p-1}}\pmod{p^4}.
\end{align*}
Obviously, if $p\mid c$, we have
\begin{align}\label{sigma3cg+}
\Sigma_3\equiv 0\pmod{p^4}.
\end{align}
Also, if $p\mid b$, we obtain
\begin{align}\label{sigma3cg++}
\Sigma_3\equiv p-p\frac{b^2}{4c}\pmod{p^4}.
\end{align}
Combining \eqref{sigma1cg+}, \eqref{sigma2cg+}, and \eqref{sigma3cg+} yields \eqref{congruence1+}. Furthermore, combining \eqref{sigma1cg++}, \eqref{sigma2cg++}, and \eqref{sigma3cg++} leads to \eqref{congruence1++}. \\
For any integer $y$ coprime with $p$, we have $y^{p-1}=1+pq_p(y)$. In addition, with the help of Fermat's Little Theorem, we show that;
\begin{align*}
\frac{1}{y^{p-1}}\equiv 1-pq_p(y)+p^2q_p(y)^2\pmod{p^3}.
\end{align*}
Consequently, for $p\nmid c$ and $p\nmid b$, we get
\begin{align*}
\Sigma_3&\equiv p+p^2\left( q_p(4c)-q_p(d)\right) +p^3\left( q_p(d)^2-q_p(d)q_p(4c)\right) \pmod{p^4}.
\end{align*} 
This, together with \eqref{sigma1cg} and \eqref{sigma2cg}, gives \eqref{congruence1}. This completes the proof.
\end{proof}

\begin{proof}[Proof of Theorem \ref{theorem2}]
With the help of identity \eqref{sunidentity}, we find that
\begin{align*}
\sum_{k=0}^{p-1}\left( 2k+1\right)^3 \frac{T_k\left( b,c\right) ^4}{d^{2k}}&=\sum_{k=0}^{p-1}\frac{\left( 2k+1\right)^3}{d^{2k}}\sum_{i=0}^{k} \sum_{j=0}^{k} \binom{2i}{i}^2\binom{2j}{j}^2\binom{k+i}{2i}\binom{k+j}{2j}d^{2k-i-j}c^{j+i}\\
&=\sum_{i=0}^{p-1} \sum_{j=0}^{p-1} \binom{2i}{i}^2\binom{2j}{j}^2\frac{c^{j+i}}{d^{i+j}}\sum_{k=\max\left( i,j\right) }^{p-1}\left( 2k+1\right)^3\binom{k+i}{2i}\binom{k+j}{2j}.
\end{align*}
By \eqref{1.2}, we obtain
\begin{align*}
\sum_{k=0}^{p-1}\left( 2k+1\right)^3 \frac{T_k\left( b,c\right) ^4}{d^{2k}}&=-\sum_{i=0}^{p-1} \sum_{j=0}^{p-1} \binom{2i}{i}\binom{2j}{j}\frac{c^{j+i}}{d^{i+j}}\frac{\vartheta_{p,i,j}\mu_{p,i,j}}{(1+i+j)(2 + i + j)}\\
&=-\sum_{i=0}^{p-1} \sum_{j=0}^{p-1} \binom{2i}{i}\binom{2j}{j}\frac{c^{j+i}}{d^{i+j}}\vartheta_{p,i,j}\mu_{p,i,j}\left( \frac{1}{1+i+j}-\frac{1}{2+i+j}\right) .
\end{align*}
For $r\in \left\lbrace 1,2\right\rbrace $, define
\begin{align*}
\Sigma_{1,r}&:=\sum_{i=0}^{p-1-r} \sum_{j=0}^{p-1-r-i} \binom{2i}{i}\binom{2j}{j}\frac{c^{j+i}}{d^{i+j}}\frac{\vartheta_{p,i,j}\mu_{p,i,j}}{r+i+j},\\
\Sigma_{2,r}&:=\sum_{i=0}^{p-r} \sum_{j=p+1-r-i}^{p-1} \binom{2i}{i}\binom{2j}{j}\frac{c^{j+i}}{d^{i+j}}\frac{\vartheta_{p,i,j}\mu_{p,i,j}}{r+i+j},\\
\Sigma_{3,r}&:=\frac{c^{p-r}}{d^{p-r}}\frac{1}{p}\sum_{i=0}^{p-r}\binom{2i}{i}\binom{2(p-r-i)}{p-r-j}\frac{c^{j+i}}{d^{i+j}}\vartheta_{p,i,p-i-r}\mu_{p,i,p-i-r}
\end{align*}
and
\begin{align*}
\Sigma_4:=\binom{2(p-1)}{p-1}^2\left(2p-1\right)\sum_{j=0}^{p-2} \binom{2j}{j}^2\frac{c^{p-1+j}}{d^{p-1+j}}\binom{p+j-1}{2j} \frac{\left( j+p\right)\mu_{p,p-1,j}}{p+1+j}\\-\binom{2(p-1)}{p-1}^4(2p-1)^4\left( \frac{c}{d}\right) ^{2p-2}.
\end{align*}
It is easy to show that
\begin{align}\label{theorem2f}
\sum_{k=0}^{p-1}\left( 2k+1\right)^3 \frac{T_k\left( b,c\right) ^4}{d^{2k}}=-\sum_{m=1}^3 \left(\Sigma_{m,1}-\Sigma_{m,2}\right)+\Sigma_4.
\end{align}
From \eqref{Lemmac13} we have $p$ divides $\binom{2(p-1)}{p-1}$, and in view of \eqref{bc11} we easily get
\begin{align}\label{s4}
\Sigma_4 \equiv 0 \pmod{p^3}.
\end{align}
Furthermore, applying Lemma \ref{lemma2} modulo $p^3$, we find that
\begin{align*}
\Sigma_{2,r}&\equiv p^2\sum_{i=0}^{p-r} \sum_{j=p+1-r-i}^{p-1} \binom{2i}{i}\binom{2j}{j}\frac{(-c)^{j+i}}{d^{i+j}}\frac{2+3k+4i(k-i)}{r+i+j} \pmod{p^3}\\
&= p^2\sum_{i=0}^{p-r}\sum_{k=p+1-r}^{p-1+i} \binom{2i}{i}\binom{2(k-i)}{k-i}\frac{(-c)^{k}}{d^{k}}\frac{2+3k+4i(k-i)}{r+k} \pmod{p^3}.
\end{align*}
Since $\binom{2j}{j}\equiv 0\pmod{p}$ for $p/2<j<p$, we deduce
\begin{align*}
\Sigma_{2,r}\equiv0\pmod{p^3}.
\end{align*}
Consequently, we reduce \eqref{theorem2f} as follows
\begin{align}\label{theorem2cgf}
\sum_{k=0}^{p-1}\left( 2k+1\right)^3 \frac{T_k\left( b,c\right) ^4}{d^{2k}}\equiv\Sigma_{1,2}-\Sigma_{1,1}+\Sigma_{3,2}-\Sigma_{3,1}\pmod{p^3}.
\end{align}
We first evaluate $\Sigma_{1,2}-\Sigma_{1,1}$ modulo $p^3$. Using Lemma \ref{lemma2}, we have
\begin{align*}
\Sigma_{1,r}&\equiv p^2\sum_{i=0}^{p-1-r} \sum_{j=0}^{p-1-r-i} \binom{2i}{i}\binom{2j}{j}\frac{(-c)^{j+i}}{d^{i+j}}\frac{2 +3i+3j+4ij}{r+i+j}\pmod{p^3}\\
&=p^2\sum_{i=0}^{p-1-r}\sum_{k=i}^{p-1-r}\binom{2i}{i}\binom{2(k-i)}{k-i}\frac{(-c)^{k}}{d^{k}}\frac{2 +3k+4i(k-i)}{r+k}\pmod{p^3}\\
&=p^2\sum_{k=0}^{p-1-r}\frac{(-c)^{k}}{d^{k}}\frac{1}{r+k}\sum_{i=0}^{k}\binom{2i}{i}\binom{2(k-i)}{k-i}\left(2+3k+4i(k-i)\right)\pmod{p^3}.
\end{align*}
By identity \eqref{identity2+}, we obtain
\begin{align*}
\Sigma_{1,r}\equiv \frac{p^2}{2}\sum_{k=0}^{p-1-r}\frac{(-4c)^{k}}{d^{k}}\frac{(1+k)(4+k)}{r+k} \pmod{p^3}.
\end{align*}
We distinguish some cases. If $p \mid c$, we find that
\begin{align*}
\Sigma_{1,r}\equiv \frac{2p^2}{r} \pmod{p^3},
\end{align*}
and hence
\begin{align}\label{s1-(1-2)+}
\Sigma_{1,2}-\Sigma_{1,1}\equiv -p^2 \pmod{p^3}.
\end{align}
If $p \mid b$, we obtain
\begin{align*}
\Sigma_{1,r}&\equiv \frac{p^2}{2}\sum_{k=0}^{p-1-r}\frac{(1+k)(4+k)}{r+k} \pmod{p^3},
\end{align*}
which gives
\begin{align*}
\Sigma_{1,1}\equiv -\frac{3p^2}{2}\pmod{p^2} \text{ and }\ \Sigma_{1,2}\equiv \frac{p^2}{2}\left(-1-2H_{p-1}\right)\pmod{p^3}.
\end{align*}
Since $H_{p-1}\equiv 0\pmod{p^2}$, we deduce that
\begin{align}\label{s1-(1-2)++}
\Sigma_{1,2}-\Sigma_{1,1}\equiv p^2 \pmod{p^3}.
\end{align}
Now, set $p\nmid c$ and $p\nmid b$. It is follows that
\begin{align*}
\Sigma_{1,1}&\equiv \frac{p^2}{2}\sum_{k=0}^{p-2} \frac{(-4c)^k}{d^k}(4+k)\pmod{p^3}\\
&\equiv p^2\frac{d}{2b^2}  \pmod{p^3}.
\end{align*}
Note that
\begin{align*}
\sum_{k=0}^{p-3}(-4x)^{k}\frac{(1+k)(4+k)}{2+k}=
\sum_{k=1}^{p-1}(-4x)^{k-2}(1+k)-2\sum_{k=1}^{p-1}(-4x)^{k-2}\frac{1}{k}.
\end{align*}
Using the above identity and with help of \eqref{cg2.9}, we obtain
\begin{align*}
\Sigma_{1,2}\equiv -p^2\frac{d^2}{8cb^2}-p^2\frac{d}{16c^2}\left(dq_p(d)+4cq_p(4c)-2b^2q_p(b)\right)  \pmod{p^3}.
\end{align*}
Therefore
\begin{align}\label{s1-(1-2)}
\Sigma_{1,2}-\Sigma_{1,1}\equiv -p^2\frac{d}{8c}-p^2\frac{d}{16c^2}\left(dq_p(d)+4cq_p(4c)-2b^2q_p(b)\right) \pmod{p^3}.
\end{align}

It remains to evaluate $\Sigma_{3,2}-\Sigma_{3,1}$ modulo $p^3$. With the help of \eqref{varformula}, we get
\begin{align*}
\Sigma_{3,r}\equiv p\left( -1\right) ^{r-1}\frac{c^{p-r}}{d^{p-r}}\sum_{i=0}^{p-r}\binom{2i}{i}\binom{2(p-r-i)}{p-r-i}\theta_{p,i,p-i-r} \pmod{p^3},
\end{align*}
where $\theta_{p,i,k}:=2+3(p-1+k)+4i(k+p-1-i)$. When $p \mid c$ we have
\begin{align*}
\Sigma_{3,r}\equiv 0\pmod{p^3}.
\end{align*}
This, together with \eqref{s1-(1-2)+} and in view of \eqref{theorem2cgf} we obtain \eqref{theorem21}. \\
Applying identity \eqref{identity2+}, we find that
\begin{align*}
\Sigma_{3,1}&\equiv p\frac{c^{p-1}}{d^{p-1}}\sum_{i=0}^{p-1}\binom{2i}{i}\binom{2(p-1-i)}{p-1-i}(-4i^2+4ip- 4i+3p-1)\pmod{p^3}\\
&\equiv \frac{3p^2}{2}\pmod{p^3}.
\end{align*}
On other hand side, we have
\begin{align*}
\Sigma_{3,2}&\equiv -p\frac{c^{p-2}}{d^{p-2}}\sum_{i=0}^{p-2}\binom{2i}{i}\binom{2(p-2-i)}{p-2-i}\left(-4i^2+4ip-8i+3p-4\right)\pmod{p^3}.
\end{align*}
Suppose that $p\mid b$. Using \eqref{identity2+} and \eqref{identity2}, we easily get
\begin{align*}
\Sigma_{3,2}\equiv -p+ \frac{p^2}{2}\pmod{p^3},
\end{align*}
and hence
\begin{align*}
\Sigma_{3,2}-\Sigma_{3,1}\equiv -p-p^2 \pmod{p^3}.
\end{align*}
Combining the above congruence with \eqref{s1-(1-2)++} and in view of \eqref{theorem2cgf} we conclude \eqref{theorem22}.\\
Now, we consider $p\nmid c$ and $p\nmid b$. By \eqref{identity2+} and \eqref{identity2}, we obtain
\begin{align*}
\Sigma_{3,2}&\equiv -\frac{dp^2}{8c}+p\frac{d}{4c}\left(1+pq_p(4c)-pq_p(d)\right)\pmod{p^3}\\
&\equiv p\frac{d}{4c}+p^2\left( q_p(4c)-q_p(d)-\frac{d}{8c}\right)\pmod{p^3}.
\end{align*}
Hence
\begin{align*}
\Sigma_{3,2}-\Sigma_{3,1}\equiv p\frac{d}{4c}+p^2\frac{d}{4c}\left(q_p(4c)-q_p(d)-\frac{12c}{2d}-\frac{1}{2}\right)\pmod{p^3}.
\end{align*}
This, together with \eqref{s1-(1-2)} and in view of \eqref{theorem2cgf} yields \eqref{theorem23}. This concludes the proof.
\end{proof}

\begin{proof}[Proof of Theorem \ref{theorem3}]Using identity \eqref{sunidentity}, we get
\begin{align*}
&\sum_{k=0}^{p-1}(2k+1)(-1)^k\frac{T_k(b,c)^4}{d^{2k}}\\
&=\sum_{i=0}^{p-1}\sum_{j=0}^{p-1}\binom{2i}{i}^2\binom{2j}{j}^2\left(\frac{c}{d}\right)^{i+j}\sum_{k=\max(i,j)}^{p-1}\binom{k+i}{2i}\binom{k+j}{2j}(2k+1)(-1)^k.
\end{align*}
Set
\begin{align*}
\Sigma_1&:=\binom{2(p-1)}{p-1}^2(2p-1)\sum_{i=0}^{p-1}\binom{2i}{i}^2\binom{p+i-1}{2i}\left(\frac{c}{d}\right)^{p-1+i}, \\
\Sigma_2&:=\sum_{i=0}^{p-1}\sum_{j=0}^{p-2}\binom{2i}{i}^2\binom{2j}{j}^2\left(\frac{c}{d}\right)^{i+j}\sum_{k=\max(i,j)}^{p-1}\binom{k+i}{2i}\binom{k+j}{2j}(2k+1)(-1)^k.
\end{align*}
It is easy to show that
\begin{align*}
\Sigma_1\equiv -p^2\pmod{p^3}.
\end{align*}
In addition, by \eqref{theorem3cg1}, we get
\begin{align*}
\Sigma_2\equiv p\sum_{i=0}^{p-1}\sum_{j=0}^{p-2}\binom{2i}{i}\binom{2j}{j}\binom{i+j}{j}\left(\frac{-c}{d}\right)^{i+j}\pmod{p^3}.
\end{align*}
By Lucas' Theorem and Fermat's Little Theorem, it follows that
\begin{align*}
p\binom{2(p-1)}{p-1}\sum_{i=0}^{p-1}\binom{2i}{i}\binom{p+j-1}{j}\left(\frac{-c}{d}\right)^{p-1+i}
\equiv -p^2 \pmod{p^3}.
\end{align*}
Thus
\begin{align*}
\sum_{k=0}^{p-1}(2k+1)(-1)^k\frac{T_k(b,c)^4}{d^{2k}}\equiv p\sum_{i=0}^{p-1}\sum_{j=0}^{p-1}\binom{2i}{i}\binom{2j}{j}\binom{i+j}{j}\left(\frac{-c}{d}\right)^{i+j}\pmod{p^3}.
\end{align*}
Since $\binom{2i}{i} \equiv 0 \pmod{p}$ for $\frac{p}{2} < i < p$ and $\binom{i+j}{j} \equiv 0 \pmod{p}$ for $p - i \leq j \leq p - 1$, then we have 
\begin{align*}
\sum_{k=0}^{p-1}(2k+1)(-1)^k\frac{T_k(b,c)^4}{d^{2k}}\equiv\ & p\sum_{i=0}^{p-1}\sum_{j=0}^{p-1-i}\binom{2i}{i}\binom{2j}{j}\binom{i+j}{j}\left(\frac{-c}{d}\right)^{i+j}\pmod{p^3}\\
=\ & p\sum_{i=0}^{p-1}\sum_{i=k}^{p-1}\binom{2i}{i}\binom{2(k-i)}{k-i}\binom{k}{i}\left(\frac{-c}{d}\right)^{k}\pmod{p^3}\\
= \ & p\sum_{k=0}^{p-1}\left(\frac{-c}{d}\right)^{k}\sum_{i=0}^{k}\binom{2i}{i}\binom{2(k-i)}{k-i}\binom{k}{i}\pmod{p^3}.
\end{align*}
Using the following binomial identity (see \cite[Lemma 4.2]{guo}):
\begin{align}\label{catalanf}
\sum_{i=0}^{k}\binom{2i}{i}\binom{2(k-i)}{k-i}\binom{k}{i}=\sum_{i=0}^{k}\binom{2i}{i}^2\binom{i}{k-i}(-4)^{k-i},
\end{align}
we obtain
\begin{align*}
\sum_{k=0}^{p-1}(2k+1)(-1)^k\frac{T_k(b,c)^4}{d^{2k}}&\equiv p \sum_{i=0}^{(p-1)/2}\binom{2i}{i}^2\sum_{k=i}^{p-1}\left(\frac{-c}{d}\right)^{k}\binom{i}{k-i}(-4)^{k-i}\pmod{p^3},\\
&\equiv p \sum_{i=0}^{(p-1)/2}\binom{2i}{i}^2\frac{(-c)^i}{d^i}\sum_{k=0}^{i}\binom{i}{k}\frac{(4c)^k}{d^k}\pmod{p^3}.
\end{align*}
Finally, applying the binomial theorem yields \eqref{theorem3-cg-1}.

Similarly, using \eqref{theorem3cg1+} and the same steps as in the proof above, we obtain modulo $ p^3 $ that
\begin{align*}
&\sum_{k=0}^{p-1}(2k+1)^3(-1)^k\frac{T_k(b,c)^4}{d^{2k}}\\
&\equiv -p\sum_{k=0}^{p-1}\left(\frac{-c}{d}\right)^{k}\sum_{i=0}^{k}\binom{2i}{i}\binom{2(k-i)}{k-i}\binom{k}{i}(3+4k+4(k-i)i)\\
&=-p\sum_{k=0}^{p-1}\left(\frac{-c}{d}\right)^{k}(3+4k)\sum_{i=0}^{k}\binom{2i}{i}\binom{2(k-i)}{k-i}\binom{k}{i}\\
&\quad-4p\sum_{k=0}^{p-1}\left(\frac{-c}{d}\right)^{k}\sum_{i=0}^{k}\binom{2i}{i}\binom{2(k-i)}{k-i}\binom{k}{i}(k-i)i.
\end{align*}
In view of \eqref{catalanf} and \eqref{catalan-new}, we obtain modulo $ p^3 $ that
\begin{align*}
&\sum_{k=0}^{p-1}(2k+1)^3(-1)^k\frac{T_k(b,c)^4}{d^{2k}}\\
&\equiv -p\sum_{i=0}^{(p-1)/2}\binom{2i}{i}^2\sum_{k=i}^{p-1}\left(\frac{-c}{d}\right)^{k}\binom{i}{k-i}(-4)^{k-i}\left( 3+4k-2(k-i)\right)\\
&= -p\sum_{i=0}^{(p-1)/2}\binom{2i}{i}^2\left(\frac{-c}{d}\right)^{i}\sum_{k=0}^{i}\binom{i}{k}\left(\frac{4c}{d}\right)^{k}\left( 3+2k+4i\right).
\end{align*}
Finally, applying the binomial theorem leads to \eqref{theorem3-cg-2}. This completes the proof.
\end{proof}

\begin{Acks}
This paper is partially supported by DGRSDT grant $n^\circ$ C0656701. The authors are grateful to the referee for reading the manuscript carefully and making helpful suggestions that increased the paper's quality.
\end{Acks}

\end{document}